\documentclass[a4paper,11pt]{amsart}
\usepackage{amsthm,amssymb,amsfonts,amsmath,latexsym,graphicx,hyperref}

\newcommand{\beq}{\begin{eqnarray}}
\newcommand{\eeq}{\end{eqnarray}}
\newcommand{\be}[1]{\begin{equation}\label{#1}}
\newcommand{\ee}{\end{equation}}

\newcommand{\eps}{{\eps}}

\newcommand{\R}{{\mathbb R}}

\newcommand{\Sp}{\mathbb S}
\def\1{\mathbb I}
\renewcommand{\(}{\left(}
\renewcommand{\)}{\right)}
\newcommand{\nrm}[2]{\|{#1}\|_{\mathrm L^{#2}(\R^N)}}
\newcommand{\nrmC}[2]{\|{#1}\|_{\mathrm L^{#2}(\mathcal C)}}
\renewcommand{\eps}{\varepsilon}
\newcommand{\gammatheta}{{\gamma_{_\theta}}}

\newtheorem{theorem}{Theorem}
\newtheorem{lemma}[theorem]{Lemma}

\newtheorem{corollary}[theorem]{Corollary}
\newtheorem{proposition}[theorem]{Proposition}

\newtheorem{remark}[theorem]{Remark}
\begin{document}

\title[Symmetry results for Caffarelli-Kohn-Nirenberg inequalities]{Symmetry of extremals of functional inequalities via spectral estimates for linear operators}

\author[J. Dolbeault, M. J. Esteban]{Jean Dolbeault, Maria J. Esteban}
\address{Jean Dolbeault, Maria J. Esteban: Ceremade (UMR CNRS no. 7534), Univ. Paris-Dauphine, Pl. de Lattre de Tassigny, 75775 Paris Cedex~16, France}
\email{dolbeaul@ceremade.dauphine.fr, esteban@ceremade.dauphine.fr}

\author[M. Loss]{Michael Loss}
\address{Michael Loss: School of Mathematics, Georgia Institute of Technology Atlanta, GA 30332, USA}
\email{loss@math.gatech.edu}

\keywords{Hardy-Sobolev inequality; Caffarelli-Kohn-Nirenberg inequality; extremal functions; Kelvin transformation; Emden-Fowler transformation; radial symmetry; symmetry breaking; rigidity; Lieb-Thirring inequalities; generalized Poincar\'e inequalities; estimates of the best constants; cylinder; Riemannian manifold; Ricci curvature \\
{\scriptsize\sl AMS classification (2010):} 26D10; 46E35; 58E35}

\date{\today}

\begin{abstract}
We prove new symmetry results for the extremals of the Caffarelli-Kohn-Nirenberg inequalities in any dimension larger or equal than $2\,$, in a range of parameters for which no explicit results of symmetry were previously known.
\end{abstract}

\maketitle
\thispagestyle{empty}

\centerline {\emph{Dedicated to Elliott Lieb on the occasion of his 80th birthday}}

%%%%%%%%%%%%%%%%%%%%%%%%%%%%%%%%%%%%%%%%%%%%%%%%%%%%%%%%%%%%%%%%%%%%%%%%%%%%%%%
%%%%%%%%%%%%%%%%%%%%%%%%%%%%%%%%%%%%%%%%%%%%%%%%%%%%%%%%%%%%%%%%%%%%%%%%%%%%%%%
\section{Introduction}\label{sect1}

The \emph{Caffarelli-Kohn-Nirenberg inequalities}~\cite{Caffarelli-Kohn-Nirenberg-84} in space dimension \hbox{$N\geq 2$} can be written as
\be{HSN}
\(\int_{\R^N}\frac{|w|^p}{|x|^{b\,p}}\;dx\)^{2/p}\leq\,C^N_{a,b}\int_{\R^N}\frac{|\nabla w|^2}{|x|^{2\,a}}\;dx\quad\forall\;w\in\mathcal D_{a,b}
\ee
with $a\leq b\leq a+1$ if $N\ge 3\,$, $a< b\leq a+1$ if $N=2\,$, and $a\neq a_c$ defined by
\[
a_c=a_c(N):=\frac{N-2}2\;.
\]
The exponent
\[
p=p(a,b):=\frac{2\,N}{N-2+2\,(b-a)}
\]
is determined by scaling considerations. Inequality~\eqref{HSN} holds in the space
\[
\mathcal D_{a,b}:=\Big\{\,w\in\mathrm L^p(\R^N,|x|^{-b}\,dx)\,:\,|x|^{-a}\,|\nabla w|\in\mathrm L^2(\R^N,dx)\Big\}
\]
and in this paper $C^N_{a,b}$ denotes the optimal constant. Typically, Inequality~\eqref{HSN} is stated with $a<a_c$ (see~\cite{Caffarelli-Kohn-Nirenberg-84}) so that the space $\mathcal D_{a,b}$ is obtained as the completion of $C_c^\infty(\R^N)\,$, the space of smooth functions in $\R^N$ with compact support, with respect to the norm $\|w\|^2=\|\,|x|^{-b}\,w\,\|_p^2+\|\,|x|^{-a}\,\nabla w\,\|_2^2\,$. Actually~\eqref{HSN} holds also for $a>a_c\,$, but in this case $\mathcal D_{a,b}$ is obtained as the completion with respect to $\|\cdot\|$ of the space $C_c^\infty(\R^N\setminus\{0\}):=\big\{w\in C_c^\infty(\R^N)\,:\,\mbox{supp}(w)\subset\R^N\setminus\{0\}\big\}\,$. Inequality~\eqref{HSN} is sometimes called the \emph{Hardy-Sobolev inequality,} as for $N > 2$ it interpolates between the usual Sobolev inequality ($a=0\,$, $b=0$) and the weighted Hardy inequalities corresponding to $b=a+1$ (see~\cite{Catrina-Wang-01}, a key paper by F.~Catrina and Z.-Q.~Wang, on this topic).

For $b=a<0\,$, $N\ge 3\,$, equality in~\eqref{HSN} is never achieved in $\mathcal D_{a,b}\,$. For $b=a+1$ and $N\ge 2\,$, the best constant in~\eqref{HSN} is given by $C_{a,a+1}^N=(a_c-a)^2$ and it is never achieved (see~\cite[Theorem 1.1, (ii)]{Catrina-Wang-01}). In contrast, for $a<b<a+1$ and $N\ge 2\,$, the best constant in~\eqref{HSN} is always achieved at some {\sl extremal\/} function $w_{a,b}\in \mathcal D_{a,b}\,$. However $w_{a,b}$ is not explicitly known unless we have the additional information that it is radially symmetric with respect to the origin. In the class of radially symmetric functions, the extremals of~\eqref{HSN} are all given (see~\cite{Chou-Chu-93,MR1731336,Catrina-Wang-01}) up to scaling and multiplication by a constant,~by
\[\label{9.1}
w^*_{a,b}(x)=\(1+|x|^{\frac{2\,(N-2-2\,a)(1+a-b)}{N-2\,(1+a-b)}}\)^{-\frac{N-2\,(1+a-b)}{2\,(1+a-b)}}\,.
\]
See~\cite{Catrina-Wang-01,Dolbeault-Esteban-Tarantello-08} for more details and in particular for a \emph{modified inversion symmetry} property of the extremal functions, based on a generalized Kelvin transformation, which relates the parameter regions $a<a_c$ and $a>a_c\,$.

In the parameter region $0\leq a<a_c\,$, $a\leq b\leq a+1\,$, if $N\ge 3\,$, the extremals are radially symmetric (see~\cite{Aubin-76, Talenti-76, Lieb-83, GMGT} and more specifically~\cite{Chou-Chu-93,MR1731336}; also see~\cite{DELT09} for a proof based on symmetrization and \cite{Dolbeault-Esteban-Tarantello-Tertikas} for an extension to the larger class of inequalities considered in Section \ref{Sec:Beyond}). On the other hand, extremals are known to be non-radially symmetric for a certain range of parameters $(a,b)$ identified first in~\cite{Catrina-Wang-01} and subsequently improved by V.~Felli and M.~Schneider in~\cite{Felli-Schneider-03}, where it was shown that in the region $a<0\,$, $a<b<b_{\,\rm FS}(a)$ with
\[\label{bfs}
b_{\,\rm FS}(a):=\frac{2\,N\,(a_c-a)}{\sqrt{(a_c-a)^2+(N-1)}}+a-a_c\;,
\]
extremals are non-radially symmetric. The proof is based on an analysis of the second variation of the functional associated to~\eqref{HSN} around the radial extremal $w^*_{a,b}\,$. Above the curve $b=b_{\,\rm FS}(a)\,$, all corresponding eigenvalues are positive and $w^*_{a,b}$ is a strict local minimum, while there is at least one negative eigenvalue if $b<b_{\,\rm FS}(a)$ and $w^*_{a,b}$ is then a saddle point. As $a\to-\infty\,$, $b=b_{\,\rm FS}(a)$ is asymptotically tangent to $b=a+1\,$.

By contrast, few symmetry results were available in the literature for $a<0\,$, and they are all of a perturbative nature. We refer the reader to~\cite{Oslo} for a detailed review of existing results. When $N\geq 3$ and for a fixed $b\in (a,a+1)\,$, radial symmetry of the extremals has been proved for $a$ close to $0$ (see~\cite{Lin-Wang-04,MR2053993}; also see~\cite[Theorem 4.8]{MR2001882} for an earlier but slightly less general result). In the particular case $N=2\,$, a symmetry result was proved in~\cite{Dolbeault-Esteban-Tarantello-08} for~$a$ in a neigbourhood of $0_-\,$, which asymptotically complements the symmetry breaking region described in~\cite{Catrina-Wang-01,Felli-Schneider-03,Dolbeault-Esteban-Tarantello-08} as \hbox{$a\to 0_-\,$}. Later, in~\cite{DELT09}, it was proved that for every $a<0$ and $b$ sufficiently close to $a+1\,$, the global minimizers are also symmetric. Due to the perturbative nature of these results, there were currently no explicit values $a < 0$ for which one knew the radial symmetry of the minimizer. For instance, up to now, it was not known whether the extremals of the inequality
\be{Casep=3}
\(\int_{\R^3} |w|^3\,dx\)^{2/3}\le C^3_{-1/2,0}\int_{\R^3} |x|\,|\nabla w|^2\,dx
\ee
were radial or not, and as a consequence, the value of $C^3_{-1/2,0}$ was also unknown.

For any $N\ge 2\,$, it has been proved in~\cite{DELT09} that, in the two-dimensional region of the parameters $a$ and $b\,$, the symmetry and symmetry breaking regions are both simply connected, and separated by a continuous curve starting at the point $(a=0,\,b=0)\,$, contained in the region $a<0\,$, \hbox{$b_{\,\rm FS}(a)<b<a+1$}. The curve $b_{\,\rm FS}$ can be parametrized by $a$ as a function of $b-a\,$, such that \hbox{$a\to-\infty$} as $b-a\to1_-\,$. It is also known from~\cite{DELT09} that the region of symmetry contains a neighborhood of the set $a<0\,$, $b=a+1$ and $a=0_-\,$, $0<b<1$ in the set $a<0\,$, $b_{\,\rm FS}(a)\le b<a+1\,$, but no explicit estimate of this neighborhood has been given yet. These results are all based on compactness arguments. Since radial symmetry is broken in certain parameter ranges,
it seems unlikely that a universal tool, like symmetrization, can be applied in the case $a<0\,$.

In this paper we determine a large region for $a<0$ where the extremals of the Caffarelli-Kohn-Nirenberg inequalities~\eqref{HSN} are radial. The result will be expressed in terms of the following function
\[
b_\star(a):=\frac{N\,(N-1)+4\,N\,(a-a_c)^2}{6\,(N-1)+8\,(a-a_c)^2}+a-a_c\;.
\]
%------------------------------------------------------------------------------
\begin{theorem}\label{main} Let $N\ge 2\,$. When $a < 0$ and $b_\star(a)\le b<a+1\,$, the extremals of the Caffarelli-Kohn-Nirenberg inequality~\eqref{HSN}
are radial and
\be{sharp}
C^N_{a,b}= |\Sp^{N-1}|^{\frac{p-2}p} {\textstyle\left[\frac{(a-a_c)^2\,(p-2)^2}{p+2}\right]^\frac{p-2}{2\,p}\!\left[\frac{p+2}{2\,p\,(a-a_c)^2}\right]\!\left[\frac 4{p+2}\right]^\frac{6-p}{2\,p}} \left[\tfrac{\Gamma\(\frac{2}{p-2}+\frac 12\)}{\sqrt\pi\;\Gamma\(\frac{2}{p-2}\)}\right]^\frac{p-2}p \kern -8pt .
\ee
\end{theorem}
%------------------------------------------------------------------------------
An elementary computation shows that
\[\textstyle
a\mapsto b_\star(a)-b_{\,\rm FS}(a)=\frac N2\left[1-\frac{a_c-a}{\sqrt{(a_c-a)^2+N-1}}-\frac{2\,(N-1)}{4\,(a_c-a)^2+3\,(N-1)}\right]
\]
is an increasing function of $a\le0\,$, so that, for any $a\le0\,$,
\[\label{diff2b}
0\le b_\star(a)-b_{\,\rm FS}(a)\le b_\star(0)-b_{\,\rm FS}(0)=\frac 1{1+N\,(N-1)}\;.
\]
Hence $a\mapsto b_\star(a)-b_{\,\rm FS}(a)$ is controlled by its value at $a=0\,$. The plots are qualitatively similar in any dimension $N\ge2\,$, and $a\mapsto b_\star(a)-b_{\,\rm FS}(a)$ uniformly converges for $a\le0$ to $0$ as $N$ increases.
%------------------------------------------------------------------------------
\begin{figure}[hb]
\includegraphics[width=10cm]{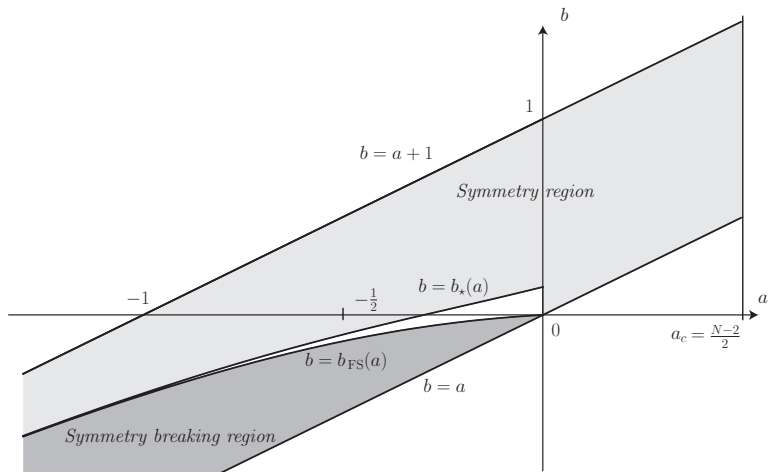}
\caption{\small\emph{In dimension $N=3\,$, the symmetry region is shown in light grey: the best constant in~\eqref{HSN} is achieved among radial functions if either $a<0$ and $b_\star(a)\le b<a+1\,$, or $0\le a<a_c$ and $a\le b\le 1\,$. The known symmetry breaking region appears in dark grey and contains at least the region $a<0\,$, $a\le b<b_{\,\rm FS}(a)\,$. Symmetry in the tiny (white) zone in between those regions is still an open question.}}
\end{figure}
%------------------------------------------------------------------------------
See Fig.~1 for an illustration of Theorem~\ref{main}. As an example, for all $N\ge 2\,$, for $a=-\frac12$ and $b=0\,$, we find that
\[
b_\star(-\tfrac 12)=-\frac 12\,\frac{N-2}{N+2}\le0\;,
\]
so that all extremals of the corresponding Caffarelli-Kohn-Nirenberg inequality, \eqref{Casep=3}, are radial and, as a consequence,
\[
\frac 1{C^N_{-1/2,0}}=\inf_{w\in\mathcal D_{a,b}(\R^N)\setminus\{0\}}\frac{\int_{\R^N} |x|\,|\nabla w|^2 dx}{\(\int_{\R^N} |w|^{\frac{2\,N}{N-1}} dx\)^{\frac{N-1}{N}}}=\tfrac 4{N\,(N-1)}\left[\tfrac{\pi^\frac{N-1}2\,\Gamma\big(N+\frac 12\big)}{\Gamma\big(\frac N2\big)\,\Gamma(N)}\right]^\frac 1N.
\]

It has been already observed in \cite{Catrina-Wang-01} that Caffarelli-Kohn-Nirenberg inequalities on $\R^N$ are equivalent to interpolation inequalities of Gagliardo-Nirenberg-Sobolev type on the cylinder $\mathcal C:=\R\times\Sp^{N-1}\,$. On the other hand, a classical observation in the Euclidean space is that Gagliardo-Nirenberg-Sobolev interpolation inequalities are equivalent to estimates on eigenvalues of a Schr\"odinger type operator, in terms of a Lebesgue norm of the potential: see \cite{Lieb-Thirring76,DFLP}. Such estimates are known as Lieb-Thirring estimates. In the present context
we shall consider only the \emph{one-bound state version of the Lieb-Thirring inequality,} in which only the lowest eigenvalue is taken into account. The strategy of the proof of Theorem \ref{main} is to exploit a similar equivalence of the two inequalities on the cylinder, and rely on the one-dimensional version of the Lieb-Thirring inequality (\emph{i.e.} which depends only on the $s$ coordinate corresponding to the axis of the cylinder) and on a generalized Poincar\'e inequality on $\Sp^{N-1}\,$.

\medskip This paper is organized as follows. In Section~\ref{cylinder}, Inequality~\eqref{HSN} is rewritten on the cylinder $\mathcal C=\R\times\Sp^{N-1}$ and the one-dimensional, one-bound state version of the Lieb-Thirring inequality is stated. Section~\ref{Sec:Proof} is devoted to the proof of Theorem~\ref{main}, which, as mentioned above, relies on the one-bound state Lieb-Thirring inequality and on a generalized Poincar\'e inequality on the sphere. Theorem~\ref{main} is also reformulated as a rigidity result for the eigenfunction associated to the lowest eigenvalue of a Schr\"odinger operator on the cylinder. Rigidity means that the optimizing potential for the Lieb-Thirring inequality depends only on the variable $s$ and not the angular variable.
Details will be given in Section~\ref{Sec:LTN}. In Section~\ref{Sec:Beyond}, a larger class of Caffarelli-Kohn-Nirenberg inequalities than \eqref{HSN} is considered: see \eqref{CKNtheta}. Neither rigidity nor symmetry results can be achieved in such a general case, but the method used for \eqref{HSN} still provides estimates on the optimal constants. Finally, in Section~\ref{Sec:General} we explain how the proofs of our results for the case $\mathcal C =\R \times \Sp^{N-1}$ can be adapted to general cylinders $\mathcal C=\R\times\mathcal M\,$, where $\mathcal M$ is a compact manifold without boundary and with positive Ricci curvature. Critical for this result is an extension of the sharp generalized Poincar\'e inequality to such manifolds $\mathcal M$ due to M.-F. Bidaut-V\'eron and L. V\'eron that can be found in \cite{BV-V}.

%%%%%%%%%%%%%%%%%%%%%%%%%%%%%%%%%%%%%%%%%%%%%%%%%%%%%%%%%%%%%%%%%%%%%%%%%%%%%%%
%%%%%%%%%%%%%%%%%%%%%%%%%%%%%%%%%%%%%%%%%%%%%%%%%%%%%%%%%%%%%%%%%%%%%%%%%%%%%%%
\section{Lieb-Thirring type inequalities on the cylinder}\label{cylinder}

Before we state the one-dimensional, one-bound state version of the Lieb-Thirring inequality, let us introduce a transformation which removes the weights in the Caffarelli-Kohn-Nirenberg inequalities.

%%%%%%%%%%%%%%%%%%%%%%%%%%%%%%%%%%%%%%%%%%%%%%%%%%%%%%%%%%%%%%%%%%%%%%%%%%%%%%%
\subsubsection*{ Emden-Fowler transformation}\label{Emden-Fowler}

We reformulate the Caffarelli-Kohn-Niren\-berg inequalities in cylindrical variables (see~\cite{Catrina-Wang-01}) using the Emden-Fowler transformation
\[\label{3.1}
s=\log|x|\;,\quad\omega=\frac{x}{|x|}\in\Sp^{N-1}\,,\quad u(s,\omega)=|x|^{a_c-a}\,w(x)\;.
\]
Inequality~\eqref{HSN} for $w$ is equivalent to a Gagliardo-Niren\-berg-Sobolev inequality on the cylinder $\mathcal C:=\R\times\Sp^{N-1}\,$, that is
\be{3.3}
\|u\|^2_{L^p(\mathcal C)}\leq\,C^N_{a,b}\;\(\nrmC{\nabla u}2^2+\Lambda\,\nrmC u2^2\)\,,
\ee
for any $u\in H^1(\mathcal C)\,$, with $\Lambda$ and $p$ given in terms of $a\,$, $a_c=a_c(N)$ and $N$ by
\[\label{Lambda-p}
\Lambda=(a_c-a)^2\quad\mbox{and}\quad p=\frac{2\,N}{N-2+2\,(b-a)}\;,
\]
and the same optimal constant $C^N_{a,b}$ as in~\eqref{HSN}.

It turns out to be convenient to reparametrize the problem, originally written in terms of $a$ and $b\,$, in terms of the parameters $\Lambda$ and $p\,$. Hence we shall use the notation $C(\Lambda,p,N):= C_{a,b}^N$ and define $C^*(\Lambda,p,N)$ as the best constant in~\eqref{3.3} among functions depending on $s$ only. For $a<0\,$, the conditions $b<b_{\,\rm FS}(a)$ and $b\ge b_\star(a)$ respectively become $\Lambda>\Lambda_{\rm FS}(p)$ and $\Lambda\le\Lambda_\star(p)\,$, where
\[
\Lambda_{\rm FS}(p):=4\,\frac{N-1}{p^2-4}\quad\mbox{and}\quad\Lambda_\star(p):=\frac{(N-1)\,(6-p)}{4\,(p-2)}\;.
\]
It is straightforward to check that $\Lambda_\star<\Lambda_{\rm FS}$ and $\frac{\Lambda_\star}{\Lambda_{\rm FS}}=\frac1{16}\,(6-p)(p+2)$ is a strictly decreasing function of $p\in(2,\frac{2\,N}{N-2}]$ for any $N\ge 3\,$, such that $1=\lim_{p\to 2}\frac{\Lambda_\star}{\Lambda_{\rm FS}}\ge\frac{\Lambda_\star}{\Lambda_{\rm FS}}\ge\lim_{p\to\frac{2\,N}{N-2}}\frac{\Lambda_\star}{\Lambda_{\rm FS}}=\frac{(N-1)(N-3)}{(N-2)^2}\,$. If $N=2\,$, $b\ge b_\star(a)$ means $p\le 6-\frac{16\,a^2}{1\,+\,4\,a^2}<6\,$.

Radial symmetry of $w=w(x)$ means that $u=u(s,\omega)$ is independent of~$\omega\,$. Up to multiplication by a constant, the extremal functions in the class of functions depending only on~$s\in\R$ solve the equation
\be{eq:onedim}
-u_*'' +\Lambda\,u_*= u_*^{p-1}\quad\mbox{in}\quad \R \,.
\ee
Up to translations in $s$ and multiplication by a constant, non-negative solutions of this equation are all equal to the function
\be{ustar}
u_*(s):=\frac A{\big[\cosh(B\,s)\big]^\frac 2{p-2}}\quad\forall\;s\in\R\;,
\ee
with
\be{AB}
A=\tfrac12\,\Lambda\quad\hbox{and}\quad B=\tfrac12\,\sqrt{\Lambda}\,(p-2)
\ee
(see Section~\ref{Sec:Beyond} for details).

We can restate Theorem~\ref{main} in terms of the variables $p$ and $\Lambda$ as follows.
%------------------------------------------------------------------------------
\begin{theorem}\label{main-cylinder} Let $N\ge 2\,$. For all $p> 2$ and $a_c^2<\Lambda\le\Lambda_\star(p)\,$, the extremals of~\eqref{3.3} depend only on the variable $s\,$. That is, all extremals of~\eqref{3.3} are equal to $u_*\,$, up to a translation and a multiplication by a constant. \end{theorem}
%------------------------------------------------------------------------------

Note that $\Lambda\in(a_c^2,\Lambda_\star(p)]$ implies $p<6$ if $N=2$ and $p<\frac{2\,N}{N-2}$ if $N\ge 3\,$.

%%%%%%%%%%%%%%%%%%%%%%%%%%%%%%%%%%%%%%%%%%%%%%%%%%%%%%%%%%%%%%%%%%%%%%%%%%%%%
\subsubsection*{Interpolation and Lieb-Thirring inequalities}\label{Sec:LT}{}

\medskip Before we prove Theorem \ref{main-cylinder}, we state a well-known result, the Lieb-Thirring inequality for one-bound state in dimension $1\,$. We will later use it for potentials depending on the variable $s\in\R$ of the cylinder.
%------------------------------------------------------------------------------
\begin{lemma}\label{thm:lt}{\rm \cite{Lieb-Thirring76}} Let $V=V(s)$ be a non-negative real valued potential in $\mathrm L^{\gamma+\frac 12}(\R)$ for some $\gamma > 1/2$ and let $-\lambda_1(V)$ be the lowest eigenvalue of the Schr\"odinger operator $-\frac{d^2}{ds^2} -V\,$. Define
\[
c_{\rm LT}(\gamma)=\frac{\pi^{-1/2}}{\gamma -1/2}\,\frac{\Gamma(\gamma+1)}{\Gamma(\gamma+1/2)}\(\frac{\gamma-1/2}{\gamma+1/2}\)^{\gamma+1/2}\,.
\]
Then
\be{inequ:lt}
\lambda_1(V)^\gamma\le c_{\rm LT}(\gamma)\int_\R V^{\gamma +1/2}(s)\;ds \;,
\ee
with equality if and only if, up to scalings, translations and a multiplication by a positive constant,
\[
V(s)=\frac{\gamma^2-1/4}{\cosh^2(s)}=:V_0(s)\,,
\]
in which case
\[
\lambda_1(V_0)=\(\gamma-1/2\)^2\,.
\]
Furthermore, the corresponding eigenspace is generated by by
\[
\psi_\gamma(s)=\pi^{-1/4}\(\frac{\Gamma(\gamma)}{\Gamma(\gamma-1/2)}\)^{1/2}\,\big[\cosh(s)\big]^{-\gamma +1/2}\,.
\]
\end{lemma}
%------------------------------------------------------------------------------

%%%%%%%%%%%%%%%%%%%%%%%%%%%%%%%%%%%%%%%%%%%%%%%%%%%%%%%%%%%%%%%%%%%%%%%%%%%%%%%
%%%%%%%%%%%%%%%%%%%%%%%%%%%%%%%%%%%%%%%%%%%%%%%%%%%%%%%%%%%%%%%%%%%%%%%%%%%%%%%
\section{Proof of main result}\label{Sec:Proof}

We will consider functions $u=u(s,\omega)$ where the variable $s$ and $\omega$ are respectively in $\R$ and $\Sp^{N-1}\,$. By $d\omega$ we denote the uniform probability measure on $\Sp^{N-1}$ and we will denote by $L^2$ the Laplace-Beltrami operator on $\Sp^{N-1}\,$, that can be written as
\[
L^2=\sum_\alpha L_\alpha^2
\]
where the sum is over all ordered pairs $\alpha=(i,j)$\,, $1\le i<j\le N\,$, and
\[
L_\alpha u=\omega_i\,\partial_j u -\omega_j\,\partial_i u\;,
\]
where $\partial_i$ denotes the derivative along the coordinate $\omega_i$ and the sphere $\Sp^{N-1}$ is defined as the set $\left\{\omega=(\omega_i)_{i=1}^N \in \R^N\,:\, \Sigma_{i=1}^N \omega_i^2=1\right\}\,$.

We shall use the abbreviation
\[
\int_{\Sp^{N-1}} |Lu|^2\,d\omega=\sum_\alpha\int_{\Sp^{N-1}} |L_\alpha u|^2\,d\omega\;.
\]
Naturally, we have
\[
\int_{\Sp^{N-1}} |Lu|^2\,d\omega= -\int_{\Sp^{N-1}} u\,(L^2 u)\;d\omega\;.
\]
The main result of our paper goes as follows.
%------------------------------------------------------------------------------
\begin{theorem}\label{thm:mainfinal}
Let $N\ge 2$ and let $u$ be a non-negative function on $\mathcal C=\R\times\Sp^{N-1}$ that satisfies
\be{pde}
-\partial_s^2 u -L^2 u +\Lambda\,u= u^{p-1} \quad\hbox{on}\quad{\mathcal C}\;,
\ee
and consider the solution $u_*$ given by~\eqref{ustar}. Assume that
\be{assumption}
\int_{\mathcal C}|u(s,\omega)|^p\,ds\,d\omega\le\int _\R |u_*(s)|^p\,ds \;,
\ee
for some $p>2\,$. If $a_c^2<\Lambda\le\Lambda_\star(p)\,$, then for a.e.~$\omega\in\Sp^{N-1}$ and $s\in\R\,$, we have $u(s,\omega)=u_*(s-C)$ for some constant $C\,$.
\end{theorem}
%------------------------------------------------------------------------------

\begin{remark}
Up to the multiplication by a constant, an optimal function for \eqref{3.3} solves \eqref{pde} so that $C^N_{a,b}=\nrmC up^{2-p}\,$. Requiring \eqref{assumption} is then natural if we look for extremal functions. Furthermore, notice that $C^N_{a,b}$ is bounded from below by the optimal constant for the inequality restricted to symmetric functions.

The symmetry result of Theorem \ref{thm:mainfinal} is in fact a uniqueness statement for the Euler-Lagrange equations associated with the variational problem~\eqref{3.3}, under an energy condition. It has been proved in \cite{Catrina-Wang-01} that optimal functions exist for the sharp Caffarelli-Kohn-Nirenberg inequalities. Hence
Theorem~\ref{main} (or equivalently Theorem~\ref{main-cylinder}) is a consequence of the stronger Theorem~\ref{thm:mainfinal}.
\end{remark}

A few comments about the strategy of the proof are in order. Equation~\eqref{eq:onedim} can be viewed as a Schr\"odinger equation with $V_*:=u_*^{p-2}$ as a potential and~$-\Lambda$ as the smallest eigenvalue. It can be readily solved and yields, up to translations, the function $u_*$ given by~\eqref{ustar} and the potential
\[\label{sol:onedim}
V_*(s):=\frac{A^{p-2}}{|\cosh(B\,s)|^2}\quad\forall\;s\in\R\;,
\]
with $A\,$, $B$ given by \eqref{AB}.
This function can be viewed as the solution for the single bound state Lieb-Thirring inequality in Lemma~\ref {thm:lt} and Corollary~\ref{Cor:6}. Up to a translation and a multiplication by a positive constant, $V_*$ is the unique optimizing potential for Inequality~\eqref{inequ:lt} for
\[
\gamma=\frac12\,\frac{p+2}{p-2}\;,
\]
if the normalization is chosen such that the optimal eigenvalue is given by~$-\Lambda\,$. Moreover, one can check that
\be{Lambdacondition}
c_{\rm LT}(\gamma)\int_\R V_*^{\gamma+\frac12}\;ds=\Lambda^\gamma\;.
\ee

\begin{proof}[Proof of Theorem~\ref{thm:mainfinal}] We will apply a number of inequalities, for which equality cases will be achieved. Let
\[
\mathcal F[u]:=\int_{\Sp^{N-1}}\int_\R\(|\partial_s u|^2-u^p\)\,ds\,d\omega+\int_{\mathcal C} |Lu|^2\;ds\,d\omega\;.
\]
If $u$ is a solution to~\eqref{pde}, then we have
\[
\mathcal F[u]=-\Lambda\int_{\mathcal C}|u(s,\omega)|^2\,ds\,d\omega\,.
\]
\subsubsection*{First step} It relies on Lemma~\ref{thm:lt}. With $V=u^{p-2}\,$, we find that a.e. in $\omega\,$,
\be{liebthirring}
\int_\R\(|\partial_s u(s,\omega)|^2 - |u(s,\omega)|^p\)\,ds\ge-\,c_{\rm LT}(\gamma)^{1/\gamma}\(\int_\R |u(s,\omega)|^p\,ds\)^{1/\gamma} |v(\omega)|^2\,,
\ee
with
\[
v(\omega):=\sqrt{\int_\R |u(s,\omega)|^2\;ds}\;.
\]
Hence we get
\begin{multline*}\mathcal F[u]\ge -\,c_{\rm LT}(\gamma)^{1/\gamma}\int_{\Sp^{N-1}}\(\int_\R |u(s,\omega)|^p\,ds\)^{1/\gamma}|v(\omega)|^2\,d\omega\\
+\int_{\mathcal C} |Lu(s,\omega)|^2\;ds\,d\omega\;.
\end{multline*}

\subsubsection*{Second step} Since
\[
L_\alpha v(\omega)=\frac{\int_\R u(s,\omega)\,L_\alpha u(s,\omega)\;ds }{\sqrt{\int_\R u(s,\omega)^2\,ds}}\;,
\]
it follows from Schwarz's inequality that
\be{convexity}
\int_{\mathcal C} |Lu(s,\omega)|^2\;ds\,d\omega\ge\int_{\Sp^{N-1}} |Lv(\omega)|^2\,d\omega\;.
\ee

\subsubsection*{Third step} Since $\gamma >1\,$, we apply H\"older's inequality,
\begin{multline}\label{Hoelder2p}
\int_{\Sp^{N-1}}\(\int_\R |u(s,\omega)|^p\,ds\)^\frac 1\gamma|v(\omega)|^2\,d\omega\\
\le\(\int_{\mathcal C }|u(s,\omega)|^p\,ds\,d\omega\)^\frac 1\gamma\(\int_{\Sp^{N-1}}|v(\omega)|^\frac{2\,\gamma}{\gamma-1}\,d\omega\)^\frac{\gamma-1}\gamma
\end{multline}
and thus, with
\[
D:= c_{\rm LT}(\gamma)^{1/\gamma}\(\int_{\mathcal C }u^p \;ds\,d\omega\)^\frac 1\gamma\,,
\]
we obtain
\[
\mathcal F[u]\ge\int_{\Sp^{N-1}} (Lv)^2\,d\omega - D\(\int_{\Sp^{N-1}} v^{\frac{2\,\gamma}{\gamma-1}}\,d\omega\)^\frac{\gamma-1}{\gamma}=:\mathcal E[v]\;.
\]

\subsubsection*{Fourth step} The generalized Poincar\'e inequality~\cite{BV-V,Beck} states that for all $q\in\big(1,\frac{N+1}{N-3}\big]\,$,
\be{generalizedSobolev}
\tfrac{q-1}{N-1}\int_{\Sp^{N-1}} (Lv)^2\,d\omega\ge\(\int_{\Sp^{N-1}} v^{q+1}\,d\omega\)^{\frac{2}{q+1}}-\int_{\Sp^{N-1}} v^2 \,d\omega\;.
\ee
Choosing
\[\label{quu}
q+1=\frac{2\,\gamma}{\gamma-1}=2\,\frac{p+2}{6-p}\;,
\]
\emph{i.e.} $q=\frac{3\,p-2}{6-p}\,$, we arrive at
\[
\mathcal E[v]\ge\(\tfrac{N-1}{q-1} -D\)\(\int_{\Sp^{N-1}} v^{q+1}\,d\omega\)^{\tfrac{2}{q+1}} -\tfrac{N-1}{q-1}\int_{\Sp^{N-1}} v^2 \,d\omega\;.
\]
Note that $q$ is in the appropriate range, since we may indeed notice that $q\le\frac{N+1}{N-3}$ is equivalent to $p\le\frac{2\,N}{N-2}\,$.

\subsubsection*{Fifth step} Using the fact that $d\omega$ is a probability measure, by H\"older's inequality, we get
\be{Hoelder}
\(\int_{\Sp^{N-1}} v^{q+1}\,d\omega\)^{\frac{2}{q+1}}\ge\int_{\Sp^{N-1}} v^2 \,d\omega \;.
\ee
Thus, if
\[
D\le\frac{N-1}{q-1}\;,
\]
we get
\[
\(\tfrac{N-1}{q-1} -D\)\(\int_{\Sp^{N-1}} v^{q+1}(\omega)\,d\omega\)^{\frac{2}{q+1}}\ge\(\tfrac{N-1}{q-1} -D\)\int_{\Sp^{N-1}} v^2\,d\omega\;,\label{qto2}
\]
that is,
\[
\mathcal E[v]\ge-D\int_{\Sp^{N-1}} v^2\,d\omega\;.
\]
By~\eqref{Lambdacondition} and~\eqref{assumption}, we know that
\[\label{ineq2}
D= c_{\rm LT}(\gamma)^\frac 1\gamma\(\int_{\mathcal C }u^p \;ds\,d\omega\)^\frac 1\gamma\le c_{\rm LT}(\gamma)^\frac 1\gamma\(\int_{\mathcal C }u_*^p\;ds\,d\omega\)^\frac 1\gamma=\Lambda\;.
\]
Thus, if
\[
\Lambda\le\frac{N-1}{q-1}=\frac{(N-1)\,(6-p)}{4\,(p-2)}=\Lambda_\star(p)\;,
\]
then $D\le\tfrac{N-1}{q-1}$
and the chain of inequalities
\[\label{bvvtwo}
-\Lambda\int_{\mathcal C} u^2\,ds\,d\omega=\mathcal F[u]\ge\mathcal E[v]\ge -D\int_{\mathcal C} u^2\,ds\,d\omega\ge -\Lambda\int_{\mathcal C} u^2\,ds\,d\omega
\]
shows that $D=\Lambda$ and equality holds at each step.

\medskip Now, let us investigate the consequences of such equalities:
\begin{enumerate}
\item By Lemma~\ref{thm:lt}, equality in~\eqref{liebthirring} yields the existence of three functions, $A\,$, $B$ and $C$ on $\Sp^{N-1}$ such that, for a.e.~$(s,\omega)\in\mathcal C\,$,
\[
u^{p-2}(s,\omega)=\frac{\big(A(\omega)\big)^{p-2}}{\big[\cosh\big(B(\omega)\,(s - C(\omega))\big)\big]^2}\;.
\]
Note that, since $\cosh(x)$ is an even function, we may choose $B(\omega)$ to be a non-negative function.
\item Equality in~\eqref{Hoelder2p} means that $v$ is proportional to $\(\int_\R |u(s,\omega)|^p\,ds\)^{1/p}$ on the sphere, so that
$B $ does not depend on $\omega\,$.
\item Equality in~\eqref{Hoelder} means that $v$ is constant on the sphere and, as a consequence, $A^2/ B$ does not depend on $\omega\,$.
\end{enumerate}
If $\Lambda<\Lambda_\star(p)$ there must be equality in~\eqref{Hoelder2p} and in~\eqref{Hoelder} and hence $A$ and $B$ do not depend on $\omega\,$.
Thus, $u$ satisfes the equation
\[
-\partial_s^2u + \Lambda u = u^{p-1} \quad\hbox{on}\quad{\mathcal C}\,.
\]
Since $u$ satisfies \eqref{pde} we have necessarily that $L^2 u=0\,$, which means that $C$ and hence $u$ are independent of $\omega\,$.

The problem is a bit trickier if $\Lambda=\Lambda_\star(p)$ and $p\le\frac{2\,N}{N-2}\,$. In this case equality in~\eqref{Hoelder} is not required. But, recall that equality in~\eqref{convexity} means that
for some $B(\omega)\,$, $C(\omega)\,$,
\[
\partial_s\log(u^{p-2} (s,\omega))= -2\,B(\omega)\,\tanh\big(B(\omega)\,(s-C(\omega))\big) :=f(s)
\]
does not depend on the variable $\omega\,$. As $s\to\infty\,$, $f(s)$ converges to $-2\,B(\omega)$ and hence $B(\omega)$ must be constant. Since $x\mapsto\tanh(x)$ is a strictly monotone function, $C(\omega)$ is also a constant. Let $X(s):=\cosh(B\,(s-C))^{-2/(p-2)}\,$. Since $u$ solves~\eqref{pde}, we find
\begin{multline*}
0=-\partial_s^2 u -L^2 u +\Lambda\,u -u^{p-1}\\
\textstyle =A\(\frac{2\,p\,B^2}{(p-2)^2}-A^{p-2}\)X^{2\,(p-1)}+\left[A\(\Lambda-\frac{4\,B^2}{(p-2)^2}\)-L^2A\right]X^2\;,
\end{multline*}
so that $A^{p-2}=\frac{2\,p\,B^2}{(p-2)^2}$ must be constant too. This is again enough to conclude that $A, B$ and $C$ do not depend on $\omega\,$.
\end{proof}

Next we state a rigidity result which is a consequence of the proof of Theorem~\ref{main-cylinder}. The connection with Theorem~\ref{thm:mainfinal} will be made clear in Section~\ref{Sec:LTN}.
%------------------------------------------------------------------------------
\begin{corollary}\label{Cor:6} Let $N\ge 2\,$. Fix $\gamma>1$ such that $\gamma\ge\frac{N-1}2$ if $N\ge 4$ and let $q=\frac{\gamma+1}{\gamma-1}\,$. Further fix $D\le\frac{N-1}{q-1}\,$. Among all potentials $V=V(s,\omega)$ with
\[
c_{\rm LT}(\gamma)^\frac1\gamma\(\int_{\mathcal C} V^{\gamma +\frac12}\,ds\,d\omega\)^\frac1\gamma=D\;,
\]
the potential $V$ that minimizes the first eigenvalue of $-\partial_s^2-L^2-V$ on $L^2(\mathcal C, ds\,d\omega)$ does not depend on $\omega\,$. Moreover, $u=V^{(2\,\gamma-1)/4}$ is optimal for~\eqref{3.3}. \end{corollary}
%------------------------------------------------------------------------------

%%%%%%%%%%%%%%%%%%%%%%%%%%%%%%%%%%%%%%%%%%%%%%%%%%%%%%%%%%%%%%%%%%%%%%%%%%%%%%%
%%%%%%%%%%%%%%%%%%%%%%%%%%%%%%%%%%%%%%%%%%%%%%%%%%%%%%%%%%%%%%%%%%%%%%%%%%%%%%%
\section{Interpolation and one-bound state Lieb-Thirring inequalities in higher dimensions}\label{Sec:LTN}

As a straightforward consequence of their definitions, both $C(\Lambda,p,N)$ and $C^*(\Lambda,p,N)$ are monotone non-increasing functions of $\Lambda$ and we have
\be{K-Kstar}
C(\Lambda,p,N)\ge C^*(\Lambda,p,N)\;,
\ee
where
\[
C^*(\Lambda,p,N)=|\Sp^{N-1}|^{-\frac{p-2}p} {\textstyle\left[\frac{\Lambda\,(p-2)^2}{p+2}\right]^\frac{p-2}{2\,p}\!\left[\frac{p+2}{2\,p\,\Lambda}\right]\!\left[\frac 4{p+2}\right]^\frac{6-p}{2\,p}} \left[\tfrac{\Gamma\(\frac{2}{p-2}+\frac 12\)}{\sqrt\pi\;\Gamma\(\frac{2}{p-2}\)}\right]^\frac{p-2}p \kern -8pt
\]
according, \emph{e.g.,} to \cite{Chou-Chu-93,MR1731336,DELT09,DDFT}. We observe that
\[
C^*(\Lambda,p,N)=C^*(1,p,N)\,\Lambda^{-\frac{p+2}{2\,p}}\,,
\]
so that $\lim_{\Lambda\to 0_+}C^*(\Lambda,p,N)=\infty\,$. From~\cite[Theorem 1.2, (ii) and Theorem 7.6, (ii)]{Catrina-Wang-01}, we know that for any $p\in(2,\frac{2\,N}{N-2})$ if $N\ge 3\,$, and any $p>2$ if $N=2\,$,
\[
\lim_{\Lambda\to\infty}\frac{\Lambda^{a_c-\frac Np}}{C(\Lambda,p,N)}=\inf_{w\in\mathrm H^1(\R^d)\setminus\{0\}}\frac{\int_{\R^d}\(|\nabla u|^2+|u|^2\)\,dx}{\(\int_{\R^d}|u|^p\,dx\)^{2/p}}\;,
\]
so that
\[
\lim_{\Lambda\to\infty}C(\Lambda,p,N)=0\;.
\]
With these observations in hand and $\gamma=\frac 12\,\frac{p+2}{p-2}\,$, we can define
\[
\Lambda_\gamma^N(\mu):=\inf\left\{\Lambda>0\,:\,\mu^\frac{2\,\gamma}{2\,\gamma+1}=1/C(\Lambda,p,N)\right\}\;.
\]
If $N=1\,$, we observe that $C(\Lambda,p,1)=C^*(\Lambda,p,1)\,$, so that $\Lambda_\gamma^1(\mu)=\Lambda_\gamma^1(1)\,\mu$ and $\Lambda_\gamma^1(1)=C^*(1,p,N)^\frac{2\,p}{p+2}\,$.

With $\gamma=\frac 12\,\frac{p+2}{p-2}\,$, notice that the condition $p\in(2,6)$ means $\gamma\in(1,\infty)$ while the condition $p\le\frac{2\,N}{N-2}$ means $\gamma\ge\frac{N-1}2\,$.

Next, consider on $\mathcal C=\R\times\Sp^{N-1}$ the Schr\"odinger operator $-\partial_s^2-L^2-V $ and denote by $-\lambda_1(V)$ its lowest eigenvalue. We assume that $V$ is non-negative, so that $\lambda_1(V)\,$, if it exists, is non-negative. The main point of this section is that $\lambda_1(V)$ can be estimated using $\Lambda_\gamma^N(\mu)$ provided $V$ is controlled in terms of $\mu\,$. The Gagliardo-Nirenberg-Sobolev inequality \eqref{3.3} on $\mathcal C$ is equivalent to the following one-bound state version of the Lieb-Thirring inequality.
%------------------------------------------------------------------------------
\begin{lemma}\label{Lem:LT} For any $\gamma\in(2,\infty)$ if $N=1\,$, or for any $\gamma\in(1,\infty)$ such that $\gamma\ge\frac{N-1}2$ if $N\ge 2\,$, if $V$ is a non-negative potential in $\mathrm L^{\gamma+\frac 12}(\mathcal C)\,$, then the operator $-\partial^2-L^2-V$ has at least one negative eigenvalue, and its lowest eigenvalue, $-\lambda_1(V)\,$, satisfies
\be{estimatelambda1}
\lambda_1(V)\le\Lambda_\gamma^N(\mu)\quad\mbox{with}\quad\mu=\mu(V):=\(\int_{\mathcal C}V^{\gamma+\frac12}\,ds\,d\omega\)^\frac 1\gamma\,.
\ee
Moreover, equality is achieved if and only if the eigenfunction $u$ corresponding to $\lambda_1(V)$ satisfies $u=V^{(2\,\gamma-1)/4}$ and $u$ is optimal for \eqref{3.3}.\end{lemma}
%------------------------------------------------------------------------------

\begin{proof}[Proof of Lemma~\ref{Lem:LT}] Let
\[
\mathcal Q[u,V]:=\frac{\int_{\mathcal C}|\nabla u|^2\,ds\,d\omega-\int_{\mathcal C}V\,|u|^2\,ds\,d\omega}{\int_{\mathcal C}|u|^2\,ds\,d\omega}
\]
so that
\[
-\lambda_1(V)=\inf_{u\in \mathrm H^1(\mathcal C)\setminus\{0\}}\mathcal Q[u,V]
\]
is achieved by some function $u\in H^1(\mathcal C)$ such that $\nrmC u2=1\,$. Using H\"older's inequality, we find that
\[
\int_{\mathcal C}V\,|u|^2\,ds\,d\omega\le\mu^\frac{2\,\gamma}{2\,\gamma+1}\,\nrmC up^2\,,
\]
with $p=2\,\frac{2\,\gamma+1}{2\,\gamma-1}$ and $\mu=\mu(V)\,$, and equality in the above inequality holds if and only if
\be{LT-CKN:EqualityCase}
V^{\gamma+\frac 12}=\kappa\,|u|^p
\ee
for some $\kappa>0\,$, which is actually such that $\kappa\,\nrmC up^p=\mu^\gamma\,$. Then we have found that
\[
\mathcal Q[u,V]\ge\nrmC{\nabla u}2^2-\mu^\frac{2\,\gamma}{2\,\gamma+1}\,\nrmC up^2\,.
\]
By definition of $\Lambda_\gamma^N(\mu)$ and \eqref{3.3}, we get that, for all $u\,$,
\[
\mathcal Q[u,V]\ge-\Lambda_\gamma^N(\mu)\;,
\]
thus proving \eqref{estimatelambda1} and
\be{CKN-LT}
-\Lambda_\gamma^N(\mu)\;=\!\!\!\!\!\!\!\inf_{\begin{array}{c}V\in\mathrm L^{\gamma+\frac 12}(\mathcal C)\cr\int_{\mathcal C}V^{\gamma+\frac12}\,ds\,d\omega=\mu^\gamma\end{array}}\hspace*{-12pt}\inf_{u\in \mathrm H^1(\mathcal C)\setminus\{0\}}\mathcal Q[u,V]
\ee
where equality holds if and only if \eqref{LT-CKN:EqualityCase} holds and $u$ is optimal for~\eqref{3.3}. This concludes the proof. \end{proof}

%%%%%%%%%%%%%%%%%%%%%%%%%%%%%%%%%%%%%%%%%%%%%%%%%%%%%%%%%%%%%%%%%%%%%%%%%%%%%%%
\subsubsection*{ A symmetry result for the one-bound state Lieb-Thirring inequality}\label{Sec:LTsym}

It is remarkable that optimality in~\eqref{estimatelambda1} is equivalent to optimality in~\eqref{3.3}. As a non trivial consequence of the above considerations and of Theorem~\ref{main-cylinder}, symmetry results for interpolation inequalities are also equivalent to symmetry results for the one-bound state Lieb-Thirring inequality in the cylinder.
%------------------------------------------------------------------------------
\begin{corollary}\label{Cor:Lieb-Thirring} Let $N\ge 2\,$. For all $\gamma>1$ such that $\gamma\ge\frac{N-1}2$ if $N\ge 4\,$, if $V$ is a non-negative potential in $\mathrm L^{\gamma+\frac 12}(\mathcal C)$ such that
\[
\mu(V):=\(\int_{\mathcal C}V^{\gamma+\frac12}\,ds\, \mu(d\omega)\)^\frac 1\gamma\le\frac{\Lambda_\star(p)}{C^*(1,p,N)^\frac{2\,p}{p+2}}\quad\mbox{with}\quad p=2\,\frac{2\,\gamma+1}{2\,\gamma-1}\,,
\]
then, the lowest (non-positive) eigenvalue of $-\partial_s^2-L^2-V\,$, $-\lambda_1(V)\,$, satisfies
\[
\lambda_1(V)\le\Lambda_\gamma^1(1)\,\mu(V)\;.
\]
Moreover, equality in the above inequality is achieved by a potential $V$ which depends only on $s$ and $u=V^{(2\gamma-1)/4}$ is optimal for \eqref{3.3}. \end{corollary}
%------------------------------------------------------------------------------
\begin{proof} Based on the definition of $\Lambda_\gamma^N(\mu)$ and \eqref{K-Kstar}, we know that
\[
\mu^\frac{2\,\gamma}{2\,\gamma+1}=\frac 1{C(\Lambda_\gamma^N(\mu),p,N)}\le\frac 1{C^*(\Lambda_\gamma^N(\mu),p,N)}=\frac 1{C^*(1,p,N)}\, \(\Lambda_\gamma^N(\mu)\)^\frac{p+2}{2\,p}\;.
\]
We observe that $\gamma=\frac 12\,\frac{p+2}{p-2}$ means $\frac{2\,\gamma}{2\,\gamma+1}=\frac{p+2}{2\,p}$ and hence
\[
\Lambda_\gamma^N(\mu)\ge\big(C^*(1,p,N)\big)^\frac{2\,p}{p+2}\mu\;.
\]
However, there is equality in the above inequality as long as $\Lambda_\gamma^N(\mu)\le\Lambda_\star(p)$ (see \eqref{CKN-LT} and Theorem \ref{thm:mainfinal}). Since $\mu\mapsto\Lambda_\gamma^N(\mu)$ is monotone increasing, requiring $\Lambda_\gamma^N(\mu)\le\Lambda_\star(p)$ is equivalent to asking $\mu\le\Lambda_\star(p)\,C^*(1,p,N)^{-2\,p/(p+2)}\,$. Then the optimality in \eqref{3.3} is achieved among symmetric functions, by Theorem~\ref{main-cylinder}. This completes the proof. Details are left to the reader.\end{proof}

%%%%%%%%%%%%%%%%%%%%%%%%%%%%%%%%%%%%%%%%%%%%%%%%%%%%%%%%%%%%%%%%%%%%%%%%%%%%%%
%%%%%%%%%%%%%%%%%%%%%%%%%%%%%%%%%%%%%%%%%%%%%%%%%%%%%%%%%%%%%%%%%%%%%%%%%%%%%%
\section{Beyond symmetry and symmetry breaking: getting estimates for the non-radial optimal constants}\label{Sec:Beyond}

Caffarelli-Kohn-Nirenberg inequalities are actually much more general than the ones considered in Section~\ref{sect1} and in view of previous results (see for instance~\cite{DDFT,1005,Dolbeault-Esteban-Tarantello-08}) it is very natural to consider another family of interpolation inequalities, which can be introduced as follows.

Define the exponent
\[
\vartheta(p,N):=N\,\frac{p-2}{2\,p}
\]
and recall that $a_c:=\frac{N-2}2\,,$ $\Lambda(a):=(a-a_c)^2$ and $p(a,b):=\frac{2\,N}{N-2+2\,(b-a)}$. We shall also set $2^*:=\frac{2\,N}{N-2}$ if $N\ge 3$ and $2^*:=\infty$ if $N=1$ or~Ê$\,2\,$. For any $a<a_c\,$, we consider the following \emph{Caffarelli-Kohn-Nirenberg inequalities}, which were introduced in~\cite{Caffarelli-Kohn-Nirenberg-84} (also see~\cite{DDFT}):

\emph{Let $b\in(a+1/2,a+1]$ and $\theta\in(1/2,1]$ if $N=1\,$, $b\in(a,a+1]$ if $N=2$ and $b\in[a,a+1]$ if $N\ge3\,$. Assume that $p=p(a,b)\,$, and $\theta\in[\vartheta(p,N),1]$ if $N\ge2\,$. Then, there exists a finite positive constant $\mathsf K_{\rm CKN}(\theta,\Lambda,p)$ with $\Lambda=\Lambda(a)$ such that, for any $u\in C_c^\infty(\R^N\setminus\{0\})\,$,
\be{CKNtheta}
\nrm{\,|x|^{-b}\,u}p^2\le\frac{\mathsf K_{\rm CKN}(\theta,\Lambda,p)}{|\Sp^{N-1}|^\frac{p-2}p\,}\,\nrm{\,|x|^{-a}\,\nabla u}2^{2\,\theta}\,\nrm{\,|x|^{-(a+1)}\,u}2^{2\,(1-\theta)}\;.
\ee}

\noindent We denote by $\mathsf K_{\rm CKN}^*(\theta,\Lambda,p)$ the best constant among all radial functions. We recall that this constant is explicit and equal to
\[
\mathsf K_{\rm CKN}^*(\theta,\Lambda,p)=\mathsf K_{\theta,p}^*\,\Lambda^{-\frac{(2\,\theta-1)\,p+2}{2\,p}}
\]
where
\[
\mathsf K_{\theta,p}^*:=\textstyle\left[\frac{(p-2)^2}{(2\,\theta-1)\,p+2}\right]^\frac{p-2}{2\,p} \left[\frac{(2\,\theta-1)\,p+2}{2\,p\,\theta}\right]^\theta \left[\frac 4{p+2}\right]^\frac{6-p}{2\,p}\left[\frac{\Gamma\left(\frac{2}{p-2}+\frac 12\right)}{\sqrt\pi\;\Gamma\left(\frac{2}{p-2}\right)}\right]^\frac{p-2}p
\]
according to \cite[Lemma 3]{DDFT}. In the special case $\theta=1\,$, we have
\[
\mathsf K_{\rm CKN}^*(1,\Lambda,p)= \left|\Sp^{N-1}\right|^\frac{p-2}p\,C^N_{a,b}\;.
\]
Define the function $\mathfrak C$ by
\begin{multline*}
\mathfrak C(p,\theta):=\tfrac{(p+2)^\frac{p+2}{(2\,\theta-1)\,p+2}}{(2\,\theta-1)\,p+2}\,\(\tfrac{2-p\,(1-\theta)}2\)^{2\,\frac{2-p\,(1-\theta)}{(2\,\theta-1)\,p+2}}\\
\cdot\(\frac{\Gamma(\frac p{p-2})}{\Gamma(\frac{\theta\,p}{p-2})}\)^\frac{4\,(p-2)}{(2\,\theta-1)\,p+2}\,\(\frac{\Gamma(\frac{2\,\theta\,p}{p-2})} {\Gamma(\frac{2\,p}{p-2})}\)^\frac{2\,(p-2)}{(2\,\theta-1)\,p+2}\;.
\end{multline*}
Notice that $\mathfrak C(p,\theta)\ge 1$ and $\mathfrak C(p,\theta)=1$ if and only if $\theta=1\,$.
%----------------------------------------------------------------------------------------------------------
\begin{theorem}\label{thm:estimates} With the above notations, for any $N\ge 3\,$, any $p\in(2,2^*)$ and any $\theta\in[\vartheta(p,N),1)\,$, we have the estimate
\[
\mathsf K_{\rm CKN}^*(\theta,\Lambda,p)\le\mathsf K_{\rm CKN}(\theta,\Lambda,p)\le\mathsf K_{\rm CKN}^*(\theta,\Lambda,p)\,\mathfrak C(p,\theta)^\frac{(2\,\theta-1)\,p+2}{2\,p}
\]
under the condition
\be{Cdt:Lambda}
a_c^2<\Lambda\le\frac{(N-1)}{\mathfrak C(p,\theta)}\,\frac{(2\,\theta-3)\,p+6}{4\,(p-2)}\;.
\ee
\end{theorem}
%----------------------------------------------------------------------------------------------------------
Although we do not establish here a symmetry result, it is interesting to notice that a symmetry result would amount to prove that $\mathsf K_{\rm CKN}(\theta,\Lambda,p)=\mathsf K_{\rm CKN}^*(\theta,\Lambda,p)\,$, except maybe on the threshold curve in the set of parameters. Theorem~\ref{thm:estimates} does not establish such a symmetry result for $\theta<1\,$, but we recover the already known fact that $\lim_{\theta\to 1}\mathsf K_{\rm CKN}(\theta,\Lambda,p)=\mathsf K_{\rm CKN}^*(1,\Lambda,p)\,$, with an explicit estimate, in the appropriate region of the parameters. This is essentially the result of Theorem~\ref{main}.

\medskip For the convenience of the reader, we split our computations in several steps and provide some details which have been skipped in the previous sections. For instance, we give the expression of $c_{\rm LT}(\gamma)$ in Lemma~\ref{thm:lt}, which is also needed to establish the expression of $\mathfrak C(p,\theta)\,$.

%%%%%%%%%%%%%%%%%%%%%%%%%%%%%%%%%%%%%%%%%%%%%%%%%%%%%%%%%%%%%%%%%%%%%%%%%%%%%%%
\subsubsection*{1. Preliminary computations}
Consider the equation
\be{Eqn:ODE}
-(p-2)^2\,w''+4\,w-2\,p\,|w|^{p-2}\,w=0 \quad\hbox{in}\quad\R\;.
\ee
The function
\[
\overline w(s)=(\cosh s)^{-\frac2{p-2}}\quad\forall\;s\in\R
\]
is, up to translations, the unique positive solution of~\eqref{Eqn:ODE}. As a consequence, the unique positive solution of
\[
-\,\theta\,w''+\eta\,w=|w|^{p-2}\,w\quad\hbox{in}\quad\R
\]
which reaches its maximum at $s=0$ can be written as
\[
w(s)=A\,\overline w(B\,s)\quad\forall\;s\in\R
\]
with
\[
A=\(\tfrac{p\,\eta}2\)^\frac 1{p-2}\quad\mbox{and}\quad B=\tfrac{p-2}2\,\sqrt{\tfrac\eta\theta}\;.
\]
As in~\cite{DDFT}, define
\[
I_q:=\int_{\R}|\overline w(s)|^q\,ds\quad\mbox{and}\quad J_2:=\int_{\R}|\overline w'(s)|^2\;ds\;.
\]
Using the formula
\[
\int_{\R}\frac{ds}{(\cosh s)^q}=\frac{\sqrt\pi\;\Gamma\(\frac q2\)}{\Gamma\(\frac{q+1}2\)}=:f(q)\;,
\]
we can compute
\[
I_2=f\(\frac 4{p-2}\)\;,\quad I_p=f\(\frac{2\,p}{p-2}\)=f\(\frac 4{p-2}+2\)\;,
\]
and get the relations
\[
\textstyle I_2=\frac{\sqrt{\pi }\;\Gamma\(\frac{2}{p-2}\)}{\Gamma\(\frac{p+2}{2\,(p-2)}\)}\;,\quad I_p=\frac{4\,I_2}{p+2}\quad\mbox{and}\quad J_2:=\frac 4{(p-2)^2}\(I_2-I_p\)=\frac{4\,I_2}{(p+2)(p-2)}\;.
\]

%%%%%%%%%%%%%%%%%%%%%%%%%%%%%%%%%%%%%%%%%%%%%%%%%%%%%%%%%%%%%%%%%%%%%%%%%%%%%%%
\subsubsection*{2. Further preliminary computations in the case $\theta<1$}
Consider now the solution of the Euler-lagrange equation satisfied by the extremals for \eqref{CKNtheta} written in the cylinder $\mathcal C\,$, that is,
\be{Euler-theta}
-\,\theta\,\(\partial_s^2 u+L^2 u\) +\left[(1-\theta)\,t[u]+\Lambda\right] u= u^{p-1}\quad\hbox{on}\quad{\mathcal C}\;,
\ee
where
\[
t[u]:=\frac{\int_{\mathcal C}\left[(\partial_s u)^2 + (Lu)^2\right]\,ds\,d\omega}{\int_{\mathcal C} u^2\;ds\,d\omega}\;.
\]
Such an extremal function always exists for all $\theta>\vartheta(p,d)$: see \cite{1005}. The case $\theta=\vartheta(p,d)$ in the theorem will be achieved by passing to the limit.

Multiplying \eqref{Euler-theta} by $u$ and integrating on $\mathcal C\,$, we find that
\[
\int_{\mathcal C}\left[(\partial_s u)^2 + (Lu)^2 + \Lambda\,u^2\right]\,ds\,d\omega=\int_{\mathcal C} u^p\,ds\,d\omega\;.
\]
To relate $\mathsf K_{\rm CKN}(\theta,\Lambda,p)$ and $\mathsf K_{\rm CKN}^*(\theta,\Lambda,p)$ we have to compare
\begin{multline*}
\mathcal Q[u]:=\frac{\kern-2pt\(\int_{\mathcal C}\left[(\partial_s u)^2 + (Lu)^2 + \Lambda\,u^2\right]\,ds\,d\omega\)^\theta\,\(\int_{\mathcal C} u^2\,ds\,d\omega\)^{1-\theta}\kern-4pt}{\(\int_{\mathcal C} u^p\,ds\,d\omega\)^\frac 2p}\\
=\(\int_{\mathcal C} u^p\,ds\,d\omega\)^{\theta-\frac 2p}\,\(\int_{\mathcal C} u^2\,ds\,d\omega\)^{1-\theta}
\end{multline*}
where equality holds because $u$ is a solution of \eqref{Euler-theta}, with the same quantity written for $u_*\,$, an extremal for \eqref{CKNtheta} in the cylinder $\mathcal C\,$, in the class of functions depending on $s\,$. Either $\mathsf K_{\rm CKN}(\theta,\Lambda,p)=\mathsf K_{\rm CKN}^*(\theta,\Lambda,p)$ and then Theorem~\ref {thm:estimates} is proved, or the inequality
\begin{multline}\label{Assumption:thetaLess1}
\frac1{\mathsf K_{\rm CKN}(\theta,\Lambda,p)}= \mathcal Q[u]=\(\int_{\mathcal C} u^p\,ds\,d\omega\)^{\theta-\frac 2p}\,\(\int_{\mathcal C} u^2\,ds\,d\omega\)^{1-\theta}\\
\le \frac1{K^*_{\rm CKN}(\theta,\Lambda,p)}=\mathcal Q[u_*]=\(\int_{\R} u_*^p\,ds\)^{\theta-\frac 2p}\,\(\int_{\R} u_*^2\,ds\)^{1-\theta}
\end{multline}
is strict.

%%%%%%%%%%%%%%%%%%%%%%%%%%%%%%%%%%%%%%%%%%%%%%%%%%%%%%%%%%%%%%%%%%%%%%%%%%%%%%%
\subsubsection*{3. The symmetric optimal function for $\theta<1$}
The solution $u_*$ can be explicitly computed. On the one hand, it solves
\[
-\,\theta\,(u_*)''+\eta\,u_*= u_*^{p-1}\quad\hbox{in}\quad\R\;,
\]
with $\eta=(1-\theta)\,t[u_*]+\Lambda\,$. After multiplying by $u_*\,$, integrating with respect to $s\in\R$ and dividing by $\int_{\R} u_*^2\,ds\,$, we find
\[
t[u_*]+\Lambda=\frac{\int_{\R} u_*^p\,ds}{\int_{\R} u_*^2\,ds}
\]
where $u_*(s)=A\,\overline w(B\,s)\,$, for all $s\in\R\,$, has been computed in the first step of this section. From this expression, we deduce that
\[
t[u_*]=B^2\,\frac{J_2}{I_2}=\frac{p-2}{p+2}\,\frac\eta\theta\quad\mbox{and}\quad\frac{\int_{\R} u_*^p\,ds}{\int_{\R} u_*^2\,ds}=A^{p-2}\,\frac{I_p}{I_2}=\frac{2\,p\,\eta}{p+2}\;,
\]
which provides the equation
\[
\frac{p-2}{p+2}\,\frac\eta\theta+\Lambda=\frac{2\,p\,\eta}{p+2}
\]
and uniquely determines
\[
\eta=\frac{(p+2)\,\theta}{(2\,\theta-1)\,p+2}\,\Lambda\;.
\]
Recall that $A=\(\tfrac{p\,\eta}2\)^\frac 1{p-2}$ and $B=\tfrac{p-2}2\,\sqrt{\tfrac\eta\theta}\,$.

%%%%%%%%%%%%%%%%%%%%%%%%%%%%%%%%%%%%%%%%%%%%%%%%%%%%%%%%%%%%%%%%%%%%%%%%%%%%%%%
\subsubsection*{4. Collecting estimates: proof of Theorem~\ref{thm:estimates}}

As in the case $\theta=1\,$, we start by estimating the functional
\[
\mathcal F[u]:=\int_{\mathcal C}\left[(\partial_s u)^2 - u^p + (Lu)^2\right]\,ds\,d\omega
\]
from below. From Lemma~\ref{thm:lt} applied with $\gamma$ replaced by some well-chosen~$\gammatheta\,$, we get the lower bound
\[
\mathcal F[u]\ge -\,c_{\rm LT}(\gammatheta)^\frac 1\gammatheta\int_{\Sp^{N-1}}\(\int_\R u^{\theta\,p} \,ds\)^\frac 1\gammatheta\int_\R u^2\,ds\,d\omega+\int_{\mathcal C} (Lu)^2\,ds\,d\omega\;,
\]
where $\gammatheta$ is now chosen such that $(\gammatheta +\frac12)\,(p-2)=\theta\,p\,$, that is
\[
\gammatheta=\frac{(2\,\theta-1)\,p+2}{2\,(p-2)}\;.
\]
Exactly as in the proof of Theorem~\ref{thm:mainfinal}, except that $\gamma$ and $p$ are now replaced respectively by $\gammatheta$ and by $\theta\,p\,$, we find the lower bound
\begin{multline*}
\mathcal F[u]\ge\int_{\Sp^{N-1}} (Lv)^2(\omega)\,d\omega\\
- c_{\rm LT}(\gammatheta)^\frac 1\gammatheta\left[\int_{\mathcal C}|u(s,\omega)|^{\theta\,p}\,ds\,d\omega\right]^\frac 1\gammatheta\(\int_{\Sp^{N-1}} v^{q+1} (\omega)\;d\omega\)^\frac 2{q+1}
\end{multline*}
with $q+1=\frac{2\,\gammatheta}{\gammatheta-1}\,$, \emph{i.e.}
\[
q=\frac{(2\,\theta+1)\,p-2}{(2\,\theta-3)\,p+6}\;.
\]
By H\"older's inequality, we find
\[
\int_{\mathcal C}u^{\theta\,p}\,ds\,d\omega\le\(\int_{\mathcal C}u^2\,ds\,d\omega\)^\frac{(1-\theta)\,p}{p-2}\(\int_{\mathcal C}u^p\,ds\,d\omega\)^\frac{\theta\,p-2}{p-2}\,.
\]
Altogether, we have shown that
\[
\mathcal F[u]\ge\int_{\Sp^{N-1}} (Lv)^2\,d\omega-D\(\int_{\Sp^{N-1}}v^{q+1}\,d\omega\)^\frac 2{q+1}
\]
where we abbreviated
\[
D:= c_{\rm LT}(\gammatheta)^\frac 1\gammatheta{\left[\(\int_{\mathcal C}u^2\,ds\,d\omega\)^\frac{(1-\theta)\,p}{p-2}\(\int_{\mathcal C}u^p\,ds\,d\omega\)^\frac{\theta\,p-2}{p-2}\right]}^\frac1\gammatheta\;.
\]
With $\gammatheta=\frac{(2\,\theta-1)\,p+2}{2\,(p-2)}$ and using Assumption°~\eqref{Assumption:thetaLess1}, we know that
\[
D=c_{\rm LT}(\gammatheta)^\frac 1\gammatheta\,\mathcal Q[u]^\frac p{\gammatheta\,(p-2)}\le c_{\rm LT}(\gammatheta)^\frac 1\gammatheta\,\mathcal Q[u_*]^\frac p{\gammatheta\,(p-2)}=\mathfrak C(p,\theta)\,\Lambda\;,
\]
where the last equality is a definition of $\mathfrak C(p,\theta)$ (see the computation of its precise value below).

As in the case $\theta=1\,$, using again the generalized Poincar\'e inequality, if $D\le\frac{N-1}{q-1}$ and $q\le\frac{N+1}{N-3}\,$, we find that
\[
\mathcal F[u]\ge -D\int_{\Sp^{N-1}}v^2\,d\omega\;.
\]
A sufficient condition for $D\le\frac{N-1}{q-1}$ is
\[
\mathfrak C(p,\theta)\,\Lambda\le\frac{N-1}{q-1}=(N-1)\,\frac{(2\,\theta-3)\,p+6}{4\,(p-2)}
\]
which is equivalent to \eqref{Cdt:Lambda}, while the condition $\frac{(2\,\theta+1)\,p-2}{(2\,\theta-3)\,p+6}=q\le\frac{N+1}{N-3}$ amounts to $p\le\frac{2\,N}{N-2\,\theta}\,$, that is $\theta\ge\vartheta(p,N)\,$.

Because of the first equality in~\eqref{Assumption:thetaLess1}, $u$ is a minimizer of $\mathcal Q[u]$ and therefore solves~\eqref{Euler-theta}. A multiplication of the equation by $u$ and an integration on $\mathcal C$ shows that
\[
-\,\Lambda\int_{\mathcal C}u^2\,ds\,\mu(d\omega)=\mathcal F[u]\;.
\]
Hence, if $\theta\ge\vartheta(p,N)$ and Condition \eqref{Cdt:Lambda} holds, we have proved that
\[
-\,\Lambda\int_{\mathcal C}u^2\,ds\,\mu(d\omega)=\mathcal F[u]\ge -D\int_{\Sp^{N-1}}v^2\,d\omega=-D\int_{\mathcal C}u^2\,ds\,\mu(d\omega)\;,
\]
that is, $\Lambda\le D\,$, and then we have the chain of inequalities
\be{Ineq:Chain}
\Lambda\le D=c_{\rm LT}(\gammatheta)^\frac 1\gammatheta\,\mathcal Q[u]^\frac p{\gammatheta\,(p-2)}\le c_{\rm LT}(\gammatheta)^\frac 1\gammatheta\,\mathcal Q[u_*]^\frac p{\gammatheta\,(p-2)}=\mathfrak C(p,\theta)\,\Lambda
\ee
where $\mathcal Q[u]=1/\mathsf K_{\rm CKN}(\theta,\Lambda,p)$ and $\mathcal Q[u_*]=1/\mathsf K_{\rm CKN}^*(\theta,\Lambda,p)$ by~\eqref{Assumption:thetaLess1}. Recalling that $\gammatheta=\frac{(2\,\theta-1)\,p+2}{2\,(p-2)}\,$, this allows to express \eqref{Ineq:Chain} as
\begin{multline*}
\frac{c_{\rm LT}(\gammatheta)^{\frac{p-2}p}}{\mathfrak C(p,\theta)^\frac{(2\,\theta-1)\,p+2}{2\,p}}\,\Lambda^{-\frac{(2\,\theta-1)\,p+2}{2\,p}}=\mathsf K_{\rm CKN}^*(\theta,\Lambda,p)\\
\le \mathsf K_{\rm CKN}(\theta,\Lambda,p)\le c_{\rm LT}(\gammatheta)^{\frac{p-2}p}\,\Lambda^{-\frac{(2\,\theta-1)\,p+2}{2\,p}}\,,
\end{multline*}
which concludes the proof of Theorem~\ref{thm:estimates}.

%%%%%%%%%%%%%%%%%%%%%%%%%%%%%%%%%%%%%%%%%%%%%%%%%%%%%%%%%%%%%%%%%%%%%%%%%%%%%%%
\subsubsection*{5. Computation of $c_{\rm LT}(\gamma)$ in Lemma~\ref{thm:lt}}

When $\theta=1$ and $d=1\,$, equality is achieved in \eqref{Ineq:Chain} and we actually have $u=u_*\,$, up to multiplication by constants, translations and scalings. As a consequence, we can compute
\[
c_{\rm LT}(\gamma)=\frac{\Lambda^\gamma}{Q[u_*]^\frac p{p-2}}=\frac{\Lambda^\gamma}{\int_{\R}u_*^p\,ds}
\]
where $\gamma=\frac{p+2}{2\,(p-2)}\,$, \emph{i.e.}~$p=2\,\frac{2\,\gamma+1}{2\,\gamma-1}\,$. With $\theta=1$ and $\eta=\Lambda\,$, we know that $u_*(s)=A\,\overline w(B\,s)$ with $A=\big(\tfrac{p\,\Lambda}2\big)^\frac 1{p-2}$ and $B=\tfrac{p-2}2\,\sqrt\Lambda\,$, so that
\[
\int_{\R}u_*^p\,ds=\frac{A^p}B\,I_p=\frac{8\,\sqrt\pi}{p^2-4}\,\frac{\Gamma\(\frac{2}{p-2}\)}{\Gamma\(\frac{p+2}{2\,(p-2)}\)}\,\(\frac{p\,\Lambda}2\)^\frac p{p-2}\,\frac 1{\sqrt\Lambda}
\]
and hence
\[
c_{\rm LT}(\gamma)=\frac{p^2-4}{8\,\sqrt\pi}\,\frac{\Gamma\(\frac{p+2}{2\,(p-2)}\)}{\Gamma\(\frac{2}{p-2}\)}\,\(\tfrac2p\)^\frac p{p-2}
\]
where $\gamma=\frac 12\,\frac{p+2}{p-2}\,$, that is $p=2\,\frac{2\,\gamma+1}{2\,\gamma-1}\,$. All computations done, we get
\[
c_{\rm LT}(\gamma)=\(\tfrac{2\,\gamma-1}{2\,\gamma+1}\)^{\gamma-\frac 12}\,\tfrac{2\,\gamma}{2\,\gamma+1}\,\frac{\Gamma(\gamma)}{\sqrt\pi\,\Gamma(\gamma+\tfrac 12)}\;.
\]

%%%%%%%%%%%%%%%%%%%%%%%%%%%%%%%%%%%%%%%%%%%%%%%%%%%%%%%%%%%%%%%%%%%%%%%%%%%%%%%
\subsubsection*{6. Computation of $\mathfrak C(p,\theta)$}

{}From \eqref{Ineq:Chain}, we know that
\[
\mathfrak C(p,\theta)=c_{\rm LT}(\gammatheta)^\frac 1\gammatheta\,\frac 1\Lambda\,\mathcal Q[u_*]^\frac{2\,p} {(2\,\theta-1)\,p+2}\;.
\]
Using the results of Step 5 allows to compute $\mathfrak C(p,\theta)\,$. Notice that the term $\mathcal Q[u_*]^\frac{2\,p} {(2\,\theta-1)\,p+2}$ is proportional to $\Lambda\,$, so that $\Lambda$ does not enter in the expression of $\mathfrak C(p,\theta)\,$.

%%%%%%%%%%%%%%%%%%%%%%%%%%%%%%%%%%%%%%%%%%%%%%%%%%%%%%%%%%%%%%%%%%%%%%%%%%%%%%
%%%%%%%%%%%%%%%%%%%%%%%%%%%%%%%%%%%%%%%%%%%%%%%%%%%%%%%%%%%%%%%%%%%%%%%%%%%%%%
\section{Interpolation and one-bound state Lieb-Thirring inequalities on general cylinders}\label{Sec:General}

In this section we extend the results of the previous sections to the more general case of the cylinders $\R\times\mathcal M\,$, where $\mathcal M$ is a Riemannian manifold, using the results of \cite{BV-V}. For this purpose we need the following assumptions:

\emph{ $(\mathcal M,g)$ is a compact Riemannian manifold of dimension $N-1\ge2\,$, without boundary, $\Delta_g$ is the Laplace-Beltrami operator on $\mathcal M\,$, the Ricci tensor $R$ and the metric tensor $g$ satisfy $R\ge\frac{N-2}{N-1}\,(q-1)\,\lambda\,g$ in the sense of quadratic forms, with $q>1\,$, $\lambda>0$ and $q\le\frac{N+1}{N-3}\,$. Moreover, one of these two inequalities is strict if $(\mathcal M,g)$ is $\Sp^{N-1}$ with the standard metric.}

For brevity, we shall say that (H) holds if these assumptions are satisfied.
%------------------------------------------------------------------------------
\begin{theorem}\label{thm:BV-V} {\rm \cite{BV-V}} Assume that {\rm (H)} holds. If $u$ is a positive solution of
\[
\Delta_g\,u-\lambda\,u+u^q=0\;,
\]
then $u$ is constant with value $\lambda^{1/(q-1)}\,$.
\end{theorem}
%------------------------------------------------------------------------------
As a consequence (see \cite[Corollary 6.2]{BV-V}), with
\[
\mathsf D(\mathcal M,q):=\max\left\{\lambda>0\,:\,R\ge\tfrac{N-2}{N-1}\,(q-1)\,\lambda\,g\right\}\;,
\]
we get the following generalized Poincar\'e inequality, an extension of \eqref{generalizedSobolev}.
%------------------------------------------------------------------------------
\begin{proposition}\label{cor:BV-V} {\rm \cite{BV-V}} Under Assumption {\rm (H)}, if $\left|M\right|=1$ and $\mathsf D(\mathcal M,q)>0\,$, then
\[
\frac 1{\mathsf D(\mathcal M,q)}\int_{\mathcal M}|\nabla v|^2+\int_{\mathcal M}|v|^2\ge\(\int_{\mathcal M}|v|^{q+1}\)^\frac 2{q+1}\quad\forall\;v\in W^{1,1}(\mathcal M)\;.
\]
\end{proposition}
%------------------------------------------------------------------------------
With these tools in hand, we can now consider the general cylinder
\[
\mathcal C:=\R\times\mathcal M\;.
\]
Using the notations of Section~\ref{Sec:Beyond}, let $b\in[a,a+1]$ and assume that $N\ge3\,$, $p=p(a,b)\,$, and $\theta\in[\vartheta(p,N),1]\,$. Under the assumptions of Proposition \ref{cor:BV-V}, there exists a finite positive constant $\mathsf K(\theta,\Lambda,p)$ such that, for any $u\in C_c^\infty(\R^N\setminus\{0\})(\mathcal C)\,$,
\be{InterpManifold}
\(\int_{\mathcal C}|u|^p\)^\frac 2p\le\mathsf K(\theta,\Lambda,p){\(\int_{\mathcal C}|\nabla u|^2+\Lambda\int_{\mathcal C}|u|^2\)}^\theta\(\int_{\mathcal C}|u|^2\)^{1-\theta}\,.
\ee
Such an interpolation inequality is easy to establish using H\"older and Sobolev inequalities on $\mathcal C\,$. We are now in a position to state a result which generalizes Theorems~\ref{main-cylinder} and \ref{thm:estimates}.
%----------------------------------------------------------------------------------------------------------
\begin{theorem}\label{thm:estimatesManifold} Assume that {\rm (H)} holds and $\mathcal M\neq\Sp^{N-1}\,$. With the above notations and $\mathfrak C(p,\theta)$ defined as in Section~\ref{Sec:Beyond}, for any $p\in(2,2^*]$ and any $\theta\in[\vartheta(p,N),1]\,$, we have the estimate
\[
\mathsf K_{\rm CKN}^*(\theta,\Lambda,p)\le\mathsf K(\theta,\Lambda,p)\le\mathsf K_{\rm CKN}^*(\theta,\Lambda,p)\,\mathfrak C(p,\theta)^\frac{(2\,\theta-1)\,p+2}{2\,p}
\]
under the conditions $p<6$ if $N=3\,$, and
\[
a_c^2<\Lambda\le\frac 1{\mathfrak C(p,\theta)}\,\mathsf D\(\mathcal M,\frac{(2\,\theta+1)\,p-2}{(2\,\theta-3)\,p+6}\)\;.
\]
In the particular case $\theta=1\,$, $\mathfrak C(p,1)=1\,$, so that $\mathsf K(\theta,\Lambda,p)=\mathsf K_{\rm CKN}(\theta,\Lambda,p)$ and the extremals of~\eqref{InterpManifold} are equal to $u_*\,$, up to translation and multiplication by a constant.
\end{theorem}
%----------------------------------------------------------------------------------------------------------
\begin{proof}[Sketch of the proof] As in step 4 of the proof of Theorem~\ref{thm:estimates}, we can apply the generalized Poincar\'e inequality of Corollary~\ref {cor:BV-V} if $D$ is such that $D\le\mathsf D(\mathcal M,q)$ with $q=\frac{(2\,\theta+1)\,p-2}{(2\,\theta-3)\,p+6}\,$. The equality case can be handled directly using Theorem~\ref{thm:BV-V}.\end{proof}

As in Theorem~\ref{thm:mainfinal}, in the particular case $\theta=1\,$, we actually have a slightly stronger result. Assume that $N\ge 3\,$, $\mathcal M\neq\Sp^{N-1}\,$, and let $u$ be a non-negative function of $-\Delta_g u +\Lambda\,u= u^{p-1}$ on $\mathcal C\,$. If $\Lambda\le\mathsf D(\mathcal M,\frac{3p-2}{6-p})$ and $\int_{\mathcal C}|u(s,\omega)|^p\,ds\,d\omega\le\int _\R |u_*(s)|^p\,ds$ , where $u_*$ is the solution given by~\eqref{ustar}, then for a.e.~$\omega\in\mathcal M$ and $s\in\R\,$, we have $u(s,\omega)=u_*(s-C)$ for some constant $C\,$.

Also in the case of a general cylinder $\mathcal C\,$, for $\theta=1\,$, we also have results similar to the one-bound state version of the Lieb-Thirring inequality of Lemma~\ref{Lem:LT} and to Corollary~\ref{Cor:Lieb-Thirring}, that can be summarized as follows.
%------------------------------------------------------------------------------
\begin{corollary}\label{Cor:LT-M} Assume that {\rm (H)} holds, $\mathcal M\neq\Sp^{N-1}\,$, $\theta=1$ and $N\ge3\,$. For any $\gamma\in(1,\infty)$ such that $\gamma\ge\frac{N-1}2\,$, if $V$ is a non-negative potential in $\mathrm L^{\gamma+\frac 12}(\mathcal C)\,$, then the operator $-\Delta_g-V$ has at least one negative eigenvalue, and its lowest eigenvalue, $-\lambda_1(V)$ satisfies
\[
\lambda_1(V)\le\Lambda_\gamma^N(\mu)\quad\mbox{with}\quad\mu=\mu(V):=\(\int_{\mathcal C}V^{\gamma+\frac12}\,ds\,d\omega\)^\frac 1\gamma
\]
for some positive constant $\Lambda_\gamma^N(\mu)\,$. Moreover, equality is achieved if and only if the eigenfunction $u$ corresponding to $\lambda_1(V)\,$, satisfies $u=V^{(2\,\gamma-1)/4}\,$, where $u$ is optimal for \eqref{InterpManifold}.

If, additionally,
\[
\mu(V)\le\frac{\mathsf D\(\mathcal M,\frac{3\,p-2}{6-p}\)}{K^{^*}(1,p,N)^\frac{2\,p}{p+2}}\quad\mbox{with}\quad p=2\,\frac{2\,\gamma+1}{2\,\gamma-1}\;,
\]
then $\lambda_1(V)\le\Lambda_\gamma^1(1)\,\mu(V)=\Lambda_\gamma^N(\mu)$ with $\Lambda_\gamma^1(1)=K^*(1,p,N)^\frac{2\,p}{p+2}\,$, and equality in the above inequality is achieved by a potential $V$ which depends only~on~$s\,$.\end{corollary}
%------------------------------------------------------------------------------
The constant $\Lambda_\gamma^N(\mu)$ can be related to $\mathsf K(\theta,\Lambda,p)$ as in the case $\mathcal M=\Sp^{N-1}\,$.

%%%%%%%%%%%%%%%%%%%%%%%%%%%%%%%%%%%%%%%%%%%%%%%%%%%%%%%%%%%%%%%%%%%%%%%%%%%%%%%
%%%%%%%%%%%%%%%%%%%%%%%%%%%%%%%%%%%%%%%%%%%%%%%%%%%%%%%%%%%%%%%%%%%%%%%%%%%%%%%
\medskip\noindent{\small{\bf Acknowlegments.} J.D.~and M.J.E.~have been supported by the projects CBDif and EVOL of the French National Research Agency (ANR). M.J.E.~has also been partially supported by the ANR project NONAP. M.L.~has been supported in part by NSF grant DMS-0901304.}

\noindent{\small \copyright\,2011 by the authors. This paper may be reproduced, in its entirety, for non-commercial purposes.}

%%%%%%%%%%%%%%%%%%%%%%%%%%%%%%%%%%%%%%%%%%%%%%%%%%%%%%%%%%%%%%%%%%%%%%%%%%%%%%%
%\bibliographystyle{siam}\bibliography{References}\end{document}

\begin{thebibliography}{10}

\bibitem{Aubin-76}
{\sc T.~Aubin}, {\em Probl\`emes isop\'erim\'etriques et espaces de {S}obolev},
  J. Differential Geometry, 11 (1976), pp.~573--598.

\bibitem{Beck}
{\sc W.~Beckner}, {\em Sharp {S}obolev inequalities on the sphere and the
  {M}oser-{T}rudinger inequality}, Ann. of Math. (2), 138 (1993), pp.~213--242.

\bibitem{BV-V}
{\sc M.-F. Bidaut-V{\'e}ron and L.~V{\'e}ron}, {\em Nonlinear elliptic
  equations on compact {R}iemannian manifolds and asymptotics of {E}mden
  equations}, Invent. Math., 106 (1991), pp.~489--539.

\bibitem{Caffarelli-Kohn-Nirenberg-84}
{\sc L.~Caffarelli, R.~Kohn, and L.~Nirenberg}, {\em First order interpolation
  inequalities with weights}, Compositio Math., 53 (1984), pp.~259--275.

\bibitem{Catrina-Wang-01}
{\sc F.~Catrina and Z.-Q. Wang}, {\em On the {C}affarelli-{K}ohn-{N}irenberg
  inequalities: sharp constants, existence (and nonexistence), and symmetry of
  extremal functions}, Comm. Pure Appl. Math., 54 (2001), pp.~229--258.

\bibitem{Chou-Chu-93}
{\sc K.~S. Chou and C.~W. Chu}, {\em On the best constant for a weighted
  {S}obolev-{H}ardy inequality}, J. London Math. Soc. (2), 48 (1993),
  pp.~137--151.

\bibitem{DDFT}
{\sc M.~Del~Pino, J.~Dolbeault, S.~Filippas, and A.~Tertikas}, {\em A
  logarithmic {H}ardy inequality}, Journal of Functional Analysis, 259 (2010),
  pp.~2045 -- 2072.

\bibitem{1005}
{\sc J.~Dolbeault and M.~J. Esteban}, {\em Extremal functions for
  {C}affarelli-{K}ohn-{N}irenberg and logarithmic {H}ardy inequalities}, to
  appear in Proc. A Edinburgh,  (2011).

\bibitem{Oslo}
\leavevmode\vrule height 2pt depth -1.6pt width 23pt, {\em About existence,
  symmetry and symmetry breaking for extremal functions of some interpolation
  functional inequalities}, in Abel Symposia, Springer, ed., 2011, to appear.

\bibitem{DELT09}
{\sc J.~Dolbeault, M.~J. Esteban, M.~Loss, and G.~Tarantello}, {\em On the
  symmetry of extremals for the {C}affarelli-{K}ohn-{N}irenberg inequalities},
  Adv. Nonlinear Stud., 9 (2009), pp.~713--726.

\bibitem{Dolbeault-Esteban-Tarantello-08}
{\sc J.~Dolbeault, M.~J. Esteban, and G.~Tarantello}, {\em The role of {O}nofri
  type inequalities in the symmetry properties of extremals for
  {C}affarelli-{K}ohn-{N}irenberg inequalities, in two space dimensions}, Ann.
  Sc. Norm. Super. Pisa Cl. Sci. (5), 7 (2008), pp.~313--341.

\bibitem{Dolbeault-Esteban-Tarantello-Tertikas}
{\sc J.~Dolbeault, M.~J. Esteban, G.~Tarantello, and A.~Tertikas}, {\em Radial
  symmetry and symmetry breaking for some interpolation inequalities}, To
  appear in Calculus of Variations and PDE,  (2011).

\bibitem{DFLP}
{\sc J.~Dolbeault, P.~Felmer, M.~Loss, and E.~Paturel}, {\em Lieb-{T}hirring
  type inequalities and {G}agliardo-{N}irenberg inequalities for systems}, J.
  Funct. Anal., 238 (2006), pp.~193--220.

\bibitem{Felli-Schneider-03}
{\sc V.~Felli and M.~Schneider}, {\em Perturbation results of critical elliptic
  equations of {C}affarelli-{K}ohn-{N}irenberg type}, J. Differential
  Equations, 191 (2003), pp.~121--142.

\bibitem{GMGT}
{\sc V.~Glaser, A.~Martin, H.~Grosse, and W.~Thirring}, {\em A family of
  optimal conditions for the absence of bound states in a potential}, Essays in
  Honor of Valentine Bargmann, E. Lieb, B. Simon, A. Wightman Eds. Princeton
  University Press, 1976, pp.~169--194.

\bibitem{MR1731336}
{\sc T.~Horiuchi}, {\em Best constant in weighted {S}obolev inequality with
  weights being powers of distance from the origin}, J. Inequal. Appl., 1
  (1997), pp.~275--292.

\bibitem{Lieb-Thirring76}
{\sc E.~Lieb and W.~Thirring}, {\em Inequalities for the moments of the
  eigenvalues of the Schr{\"o}dinger Hamiltonian and their relation to Sobolev
  inequalities}, Essays in Honor of Valentine Bargmann, E. Lieb, B. Simon, A.
  Wightman Eds. Princeton University Press, 1976, pp.~269--303.

\bibitem{Lieb-83}
{\sc E.~H. Lieb}, {\em Sharp constants in the {H}ardy-{L}ittlewood-{S}obolev
  and related inequalities}, Ann. of Math. (2), 118 (1983), pp.~349--374.

\bibitem{MR2053993}
{\sc C.-S. Lin and Z.-Q. Wang}, {\em Erratum to: ``{S}ymmetry of extremal
  functions for the {C}affarelli-{K}ohn-{N}irenberg inequalities'' [{P}roc.\
  {A}mer.\ {M}ath.\ {S}oc.\ {132} (2004), no.\ 6, 1685--1691]}, Proc. Amer.
  Math. Soc., 132 (2004), p.~2183.

\bibitem{Lin-Wang-04}
\leavevmode\vrule height 2pt depth -1.6pt width 23pt, {\em Symmetry of extremal
  functions for the {C}affarelli-{K}ohn-{N}irenberg inequalities}, Proc. Amer.
  Math. Soc., 132 (2004), pp.~1685--1691.

\bibitem{MR2001882}
{\sc D.~Smets and M.~Willem}, {\em Partial symmetry and asymptotic behavior for
  some elliptic variational problems}, Calc. Var. Partial Differential
  Equations, 18 (2003), pp.~57--75.

\bibitem{Talenti-76}
{\sc G.~Talenti}, {\em Best constant in {S}obolev inequality}, Ann. Mat. Pura
  Appl. (4), 110 (1976), pp.~353--372.

\end{thebibliography}

\end{document}